\documentclass[11pt]{article}
\usepackage[margin=1in]{geometry}
\usepackage{marginnote,mathtools}
\usepackage{graphicx,xcolor,cite,mathtools}
\usepackage{amssymb,amsmath,amsthm,mathrsfs,paralist,bm,esint,setspace}
\RequirePackage[colorlinks,citecolor=blue,urlcolor=blue]{hyperref}

\newcommand{\R}{{\mathbb{R}}}
\newcommand{\Z}{{\mathbb{Z}}}
\newcommand{\N}{{\mathbb{N}}}
\newcommand{\E}{\mathrm{E}}

\renewcommand{\P}{\mathrm{P}}
\renewcommand{\d}{\mathrm{d}}
\newcommand{\<}{\langle}
\renewcommand{\>}{\rangle}
\newcommand{\e}{\mathrm{e}}
\newcommand{\Var}{\text{\rm Var}}
\newcommand{\lip}{\text{\rm Lip}}

\DeclareMathOperator{\Cov}{\text{\rm Cov}}

\input cyracc.def

\title{A CLT for dependent random variables, with an application to
	an infinite system of interacting diffusion processes\thanks{%
	Research supported in part by  NSF grants DMS-1811181 (D.N.) and DMS-1855439 (D.K.).}}
\author{Le Chen\\Emory University\\\texttt{le.chen@emory.edu}\\
	\and
		Davar Khoshnevisan\\University of Utah\\\texttt{davar@math.utah.edu}\\
	\and\and
		David Nualart\\University of Kansas\\\texttt{nualart@ku.edu}\\
	\and
		Fei Pu\\University of Utah\\\texttt{pu@math.utah.edu}
	}
\date{\today}

\begin{document}
\newtheorem{stat}{Statement}[section]
\newtheorem{proposition}[stat]{Proposition}
\newtheorem*{prop}{Proposition}
\newtheorem{corollary}[stat]{Corollary}
\newtheorem{theorem}[stat]{Theorem}
\newtheorem{lemma}[stat]{Lemma}
\theoremstyle{definition}
\newtheorem{definition}[stat]{Definition}
\newtheorem*{cremark}{Remark}
\newtheorem{remark}[stat]{Remark}
\newtheorem*{OP}{Open Problem}
\newtheorem{example}[stat]{Example}
\newtheorem{nota}[stat]{Notation}
\numberwithin{equation}{section}
\maketitle

\begin{abstract}
	We present a central limit theorem for stationary random fields that
	are short-range dependent and asymptotically independent. As an application,
	we present a central limit theorem for an infinite family of interacting It\^o-type diffusion
	processes.
\end{abstract}

\bigskip\bigskip


Short-range dependence. Asymptotic independence. Interacting diffusions.

{
}

\section{Introduction}
Let $\zeta=\{\zeta_k\}_{k\in\Z^d}$ be a stationary random field, centered and
normalized so that
$\E(\zeta_0)=0$ and $\Var(\zeta_0)< \infty$. The principal aim of this article is to: (1)
Prove the following central limit theorem for dependent random variables; and
(2) Present an application of this theorem to interacting diffusion processes.

\begin{theorem}\label{th:clt:Intro}
	If the random field $\zeta=\{\zeta_k\}_{k\in\Z^d}$ is short-range dependent $[$see
	\eqref{SRD}$]$,
	asymptotically independent $[$see Condition {\bf (AI)}$]$,
	and satisfies a minimal integrability condition $[$see \eqref{E:UI}$]$,
	then the distribution of $n^{-d/2}\sum_{k\in\{1,\ldots,n\}^d}\zeta_k$
	is asymptotically normal as $n\to\infty$.
\end{theorem}
Short-range dependence and asymptotic independence will be recalled in due time. For now,
it suffices to say that these are both natural conditions and both arise abundantly in the literature on
time-series analysis (see for example Lahiri \cite{Lahiri} and its references).

Our main application of
Theorem \ref{th:clt:Intro} is a result about infinite systems of interacting diffusions.
Before we describe that application, let us make two brief remarks that
explain how Theorem \ref{th:clt:Intro} might relate to parts
of a vast body of  central limit theorems
that already exist  in the literature.

\begin{remark}
	One can prove that if $\zeta$ is strongly mixing in the sense
	of Rosenblatt \cite{Rosenblatt}, then $\zeta$ is asymptotically independent;
	see Corollary 1.12 of Bradley \cite[Vol.\ 1]{Bradley} for example when $d=1$. It
	follows from this that Theorem \ref{th:clt:Intro} implies a well-known central limit theorem of
	Ibragimov \cite{Ibragimov} for strongly mixing sequences though, usually, the latter is cast when
	$d=1$.  A noteworthy difference between the proof of Theorem \ref{th:clt:Intro} and the proofs of
	CLTs for strongly mixing sequences is that the proof of Theorem \ref{th:clt:Intro} is non technical,
	and relies only on compactness arguments together with Paul L\'evy's classical characterization of
	standard Brownian motion
	as the unique mean-zero, continuous L\'evy process with variance one at time one
	(this is in fact an immediate consequence of the L\'evy-Khintchine formula;
	see, for example Bertoin \cite{Bertoin}).
	The detailed bibliography of the three-volume book
	by Bradley \cite{Bradley} contains a very large number of
	pointers to the vast literature on CLT for stationary dependent sequences.
	See also the  survey articles by Bradley \cite{Bradley_PS}
	and Merlev\`ede, Peligrad, and Utev \cite{MerlevedePeligradUtev}.
\end{remark}

\begin{remark}
	It is  easy to see that if $\zeta$ is an associated sequence of random variables
	in the sense of Esary, Proschan, and Walkup \cite{EsaryProschanWalkup}, then Theorem
	\ref{th:clt:Intro} reduces to the central limit theorem of Newman and Wright
	\cite{NewmanWright} for associated random variables (again, usually stated
	for $d=1$).
\end{remark}

The preceding remarks describe how Theorem \ref{th:clt:Intro} reduces to well-known theorems
in specific settings. Next we describe a setting where we do not know whether there is
association or strong mixing.

Consider the following infinite system of interacting It\^{o}-type stochastic differential equations:
\begin{equation}\label{SHE}\begin{split}
	&\d u_t(x) =  (Lu_t)(x)\,\d t
		+ \Phi(u_t(x))\,\d B_t(x)\qquad\text{for all $t>0$ and $x\in\Z^d$},\\
	&\text{subject to }u_0(x)=1\qquad\text{for all $x\in\Z^d$},
\end{split}\end{equation}
for a field $B=\{B(x)\}_{x\in\Z^d}$ of i.i.d.\ one-dimensional Brownian motions.
Here, $L$ denotes the generator of a continuous-time random walk on $\Z^d$,
and the diffusion coefficient $\Phi: \R \mapsto \R$ is assumed to be Lipschitz continuous.
Shiga and Shimizu \cite{SS} have shown that \eqref{SHE} has a unique adapted solution
under these conditions.

\begin{theorem}\label{CLT:SHE}
	For every Lipschitz-continuous function $g$ and all $t \geq 0$,
	\begin{align}\label{CLT}
		\left\{ n^{-d/2}\sum_{x\in\Z^d}\left(g(u_t(x))- \E[g(u_t(0)]\right)
		\varphi(x/n);\,\varphi\in\mathscr{C}\right\}
		\xrightarrow{\text{\rm fdd}} \{\sigma_{g,t} 		W(\varphi);\,\varphi\in\mathscr{C}\},
	\end{align}
	where  $\sigma_{g,t}^2= \sum_{x\in\Z^d}\Cov[g(u_t(0))\,,g(u_t(x))]$ is finite, in fact absolutely convergent.
\end{theorem}

In the above theorem, and throughout the paper,
``$\xrightarrow{\text{\rm fdd}}$'' refers to the weak convergence of all finite-dimensional distributions.
Moreover $\mathscr{C}$ denotes the class of all
piecewise-continuous functions with compact support on $\mathbb{R}^d$.\footnote{Piecewise continuous
	means functions which are continuous except on a finite number of  hyperplanes.}

Earlier, Deuschel \cite{Deuschel} proved a version of Theorem \ref{CLT:SHE} --
where $L$ is replaced by a more general
nonlinear operator and $\Phi$ a less general function --
using a martingale central limit theorem
in place of Theorem \ref{th:clt:Intro}. The present formulation of Theorem \ref{CLT:SHE}
includes the important special
case of the ``parabolic Anderson model.'' That is when $\Phi(z)=\text{const}\times z$ for all $z\in\R$;
see Carmona and Molchanov \cite{CM}.

Theorem \ref{CLT:SHE}
can also be  generalized to cover many other noise models; see
Ref.\ \cite{CKNP,CKNP_b,CKNP_c} for analogous results
in the continuous setting of SPDEs. Here, we will not study such more general results to maintain brevity.

The remainder of this paper is devoted to the proof of
Theorems \ref{th:clt:Intro} and \ref{CLT:SHE}, which are proved in Sections \ref{S:CLT} and
\ref{S:SPDE}, respectively.

We close the Introduction with a brief description of the notation of this paper.
Throughout we write $\|Z\|_k$ instead of
$(\E[|Z|^k])^{1/k}$. We denote
by $\lip_f$ the Lipschitz constant of every function  $f:\R\to\mathbb{C}$; that is,
\[
	\lip_f := \sup_{-\infty<a<b<\infty}
	\frac{|f(b)-f(a)|}{|b-a|}.
\]
Rademacher's theorem (see Federer \cite[Theorem 3.1.6]{Federer})
ensures that if $\lip_f<\infty$, then the weak derivative
of $f$ exists and is bounded almost everywhere in modulus by $\lip_f$.

\section{A CLT for dependent variables}\label{S:CLT}
In this section we first recall the undefined
terminology of Theorem \ref{th:clt:Intro}, and then state and prove a more general
theorem   (Theorem \ref{th:CLT}) than Theorem \ref{th:clt:Intro} which turns out to be easier to
prove directly.

\subsection{Two definitions}\label{sec:Def}

We start with a general definition.

\begin{definition} \label{def1}
A sequence of $N$-dimensional  random vectors $(X_{1,n}, \dots, X_{N,n})$, $N\ge 2$, is {\it asymptotically independent}
 as $n\rightarrow \infty$ if for all   real numbers $\xi_j$, $1\le j\le N$,
\[
		\lim_{n\to\infty}\left |  \E\left[ \prod_{j=1}^n \e^{i\xi_jX_{j,n} }\right] -
		\prod_{j=1}^N\E\left[\e^{i\xi_j X_{j,n}}\right]\right|
		= 0.
\]
\end{definition}

Recall that a stationary random field $\zeta=\{\zeta_k\}_{k\in\Z^d}$ is said to be \emph{short-range
dependent} when
\begin{equation}\label{SRD}
	\bar{\sigma}^2:= \sum_{k\in\Z^d} |\Cov(\zeta_0\,,\zeta_k)|<\infty.
\end{equation}

It will be helpful to first introduce some notation before we
define the remaining undefined term in Theorem \ref{th:clt:Intro}.
We endow the collection
$\mathscr{C}$ of all piecewise-continuous functions
$\varphi:\R^d\to\R$ that have compact support with
the $L^2(\R^d)$ norm.
Let  $\mathscr{U} \subset \mathscr{C}$ denote the
collection of all linear combinations of indicator functions of
upright boxes
$Q\subset\R^d$ of the form $Q = (a_1\,,b_1] \times\cdots\times (a_d\,,b_d],$
where $a_1<b_1,\ldots,a_d<b_d$ are real numbers.
We always let $\text{\rm supp}[\varphi]$  denote the support of
the function $\varphi\in\mathscr{C}$.
Moreover, the \emph{separation} of two  sets $A$ and $B$ is defined as
\begin{align}\label{sep}
	{\rm sep}(A\,, B)= \inf
	\left\{\max_{1\le i\le d}|x_i- y_i|: x \in A,\,  y \in B\right\}
	\qquad \text{for $A,B \subset \R^d$}.
\end{align}

For every $\varphi\in\mathscr{C}$ we consider the sequence of random variables
\begin{equation}\label{S_n}
	S_n(\varphi) :=  n^{-d/2} \sum_{k\in\Z^d} \zeta_k\, \varphi(k/n)
	\qquad[n\in\mathbb{N}].
\end{equation}
Now we are ready to introduce the assumption of  asymptotic independence.

\begin{itemize}
\item [(\textbf{AI})]
	For every $\delta >0$ and  all $\varphi_1, \varphi_2 \in \mathscr{U}$ such that
	${\rm sep}({\rm supp}[\varphi_1]\,,{\rm supp}[\varphi_2]) \geq \delta,$
	the random variables
	$S_n(\varphi_1)$ and $S_n(\varphi_2)$ are {\em asymptotically independent} as $n\to\infty$;
	that is,
	\begin{equation}\label{AI}
		\lim_{n\to\infty}\left|\E\left[\e^{iaS_n(\varphi_1) + ibS_n(\varphi_2)}\right] -
		\E\left[\e^{iaS_n(\varphi_1)}\right]\E\left[\e^{ibS_n(\varphi_2)}\right]\right|
		= 0\qquad \text{for all $a, b \in \R$}.
	\end{equation}
\end{itemize}

We conclude this section with a CLT for dependent variables.
Let $W=\{W(\varphi);\, \varphi\in L^2(\R^d)\}$ denote the
usual isonormal Gaussian process that is associated with white noise on $\R^d$. That is,
$W$ is a mean-zero Gaussian process with
\[
	\Cov[W(\varphi_1)\,,W(\varphi_2)]=\<\varphi_1\,,\varphi_2\>_{L^2(\R^d)}
	\qquad\text{for every $\varphi_1,\varphi_2\in L^2(\R^d)$}.
\]

\begin{theorem}\label{th:CLT}
	If $\zeta=\{\zeta_k\}_{k\in\Z^d}$ is a stationary random field that satisfies \eqref{SRD}and
	\emph{(\textbf{AI})}, and if
	\begin{align} \label{E:UI}
		n\mapsto S^2_n(\bm{1}_{(0,1]^d}) 
		= n^{-d}\bigg( \sum_{k\in\{1,\ldots,n\}^d}\zeta_k\bigg)^2
		\text{is uniformly integrable},
	\end{align}
	then
	$\{S_n(\varphi);\, \varphi\in\mathscr{C}\}
	\xrightarrow{\text{\rm fdd}}\{\sigma W(\varphi);\,\varphi\in\mathscr{C}\}$
	as $n\to\infty$, where
	\begin{equation}\label{sigma}
		\sigma^2 := \sum_{k\in\Z^d}\Cov(\zeta_0\,,\zeta_k).
	\end{equation}
\end{theorem}

Let   $\varphi = \bm{1}_{(0, 1]^d}$
in order to
deduce the following precise form of Theorem \ref{th:clt:Intro} from Theorem \ref{th:CLT}.

\begin{corollary}
	If $\zeta=\{\zeta_k\}_{k\in\Z^d}$ is a stationary random field that
	satisfies \eqref{SRD}, \emph{(\textbf{AI})}, and \eqref{E:UI}, then $n^{-d/2} \sum_{k\in\{1,\ldots,n\}^d} \zeta_k
	\xrightarrow{\text{\rm d\!}} {\rm N}(0\,, \sigma^2)$
	as $n \to \infty$, where $\sigma^2$ is defined in \eqref{sigma}.
\end{corollary}

\begin{lemma}\label{UIequiv}
	Condition \eqref{E:UI} implies that 
	$n\mapsto S^2_n(\bm{1}_Q) $ is uniformly integrable for every 
	$Q\in\mathscr{U}$.
\end{lemma}

\begin{proof}
	A finite sum of uniformly integrable (denoted by UI) random variables is UI. Therefore,
	by stationarity, it suffices to prove that $n\mapsto S^2_n(\bm{1}_Q)$ is 
	UI for $Q= \prod_{k=1}^d (0\,, r_k]$, where $0<r_k\leq 1$. 
	We will prove this for $Q= (0\,, r]\times (0\,, 1]^{d-1}$ since this is all we will need later on;
	the same arguments can be applied to prove  the uniform integrability for 
	more general $Q\in\mathscr{U}$. 
	To simplify the notation, we assume without loss of much generality, that $d=2$.

	Choose and fix some $r\in(0\,,1)$, and define $Q(\alpha) := (0\,, \alpha]\times (0\,,1]$ for 
	every $\alpha\in(0\,,1]$.
	One can check from first principals that if 
	$\{S^2_{n}(\bm{1}_{Q(\alpha)})\}_{n\in \N}$ is UI for some $\alpha\in(0\,,1]$, 
	then $\{S^2_{n}(\bm{1}_{\frac{1}{2}Q(\alpha)})\}_{n\in \N}$ is UI also. 
	Hence, it follows from  condition \eqref{E:UI} and stationarity that 
	$n\mapsto \{S^2_{n}(\bm{1}_{(0, 1/2] + (0, 1/2]^2})\}_{n\in \N}$ is UI.
	And because $Q(1/2)$ is the disjoint union of $(0\,, 1/2]^2$ and
	$(0\,, 1/2)+(0\,, 1/2]^2$, it follows that $n\mapsto S^2_n(\bm{1}_{Q(1/2)})$ is UI. 
	By induction, we may deduce that  $n\mapsto S^2_n(\bm{1}_{Q(1/2^k)})$ is UI for any $k\in \N$. 
	Again, we use the fact that a finite sum of UI random variables is UI to see that
	$n\mapsto S^2_n(\bm{1}_{Q(q)})$ is UI for every real number $q\in (0\,, 1)$ 
	of finite dyadic expansion; that is, $q$ of the form $\sum_{i=1}^{m}x_i/2^i$
	where $x_1,\ldots,x_m\in \{0\,, 1\}$.

	For every $\varepsilon >0$ there exists a  number 
	$q\in (0\,, 1)$, with a finite dyadic expansion, such that $\|\bm{1}_{Q(|r-q|)}\|_{L^2(\R^2)}<\varepsilon$. 
	Thus, the following is valid for all $K,L>0$ and $n\in\N$:
	\begin{align*}
		&\E\left[S^2_{n}(\bm{1}_{Q(r)})~;~|S_n(\bm{1}_{Q(r)})| >L \right]
			= \E\left[ \left| S_{n}(\bm{1}_{Q(q)}) + S_{n}(\bm{1}_{Q(r)} - \bm{1}_{Q(q)})
			\right| ^2~;~ |S_n(\bm{1}_{Q(r)}) |>L \right]\\
		&\quad \leq 2\E\left [S^2_{n}(\bm{1}_{Q(q)})~;~
			|S_n(\bm{1}_{Q(r)}) |>L \right] +  2\E\left[ S^2_{n}(\bm{1}_{Q(|r-q|)}) \right] \\
		&\quad \leq 2\E\left[ S^2_{n}(\bm{1}_{Q(q)})
			~;~ |S_n(\bm{1}_{Q(q)})| >K \right] + 
			\frac{2K^2}{L^2} \sup_{n\in \N}\E\left[S^2_{n}(\bm{1}_{Q(r)})
			\right] +  2\E\left[ S^2_{n}(\bm{1}_{Q(|r-q|)})\right],
	\end{align*}
	thanks to Chebyshev's inequality.  We borrow in advance from Lemma \ref{variance:limit}
	below -- see \eqref{UI} -- to see that\footnote{Lemma \ref{variance:limit} and its proof
	do not refer to the present lemma. So this application of Lemma \ref{variance:limit} is
	logically sound.}
	\begin{align*}
		\adjustlimits
		\lim_{L\to\infty}\limsup_{n\to\infty}
		\E \left[ S^2_{n}(\bm{1}_{Q(r)})~;~
		|S_n(\bm{1}_{Q(r)})| >L \right] \leq 2\limsup_{n\to\infty}
		\E\left[S^2_{n}(\bm{1}_{Q(q)}) ~;~ |S_n(\bm{1}_{Q(q)})| >K \right] 
		+2\sigma^2\varepsilon^2
	\end{align*}
	Let $K\to \infty$, using the fact that $n\mapsto S_n^2(\bm{1}_{Q(q)})$ is UI,
	to conclude the uniform integrability of $n\mapsto S^2_n(\bm{1}_{Q(r)})$ 
	from the fact that $\varepsilon$ is arbitrary. 
\end{proof}

\subsection{Tightness and weak convergence}
Before we prove Theorem \ref{th:CLT}, we make some comments, by way of three lemmas, about
tightness and weak convergence. The proof of Theorem \ref{th:CLT} will be carried out in the
next subsection.

\begin{lemma}\label{variance:limit}
	For every $\varphi\in\mathscr{C}$ and
	all $n\in\mathbb{N}$,
	\begin{equation}\label{UI}
   		\E\left(|S_n(\varphi)|^2\right) \leq \bar{\sigma}^2 n^{-d} \sum_{k\in\Z^d}
		|\varphi(k/n)|^2
		\quad\text{and}\quad
		\lim_{n\to\infty}\E\left(|S_n(\varphi)|^2\right) =  \sigma^2\|\varphi\|_{L^2(\R^d)}^2,
	\end{equation}
	where $\sigma$ and $\bar{\sigma}$
	are defined respectively in \eqref{sigma} and \eqref{SRD}.
\end{lemma}

\begin{proof}
	A change of variables shows that
	\[
		\E\left(|S_n(\varphi)|^2\right) = n^{-d} \sum_{j\in\Z^d}
		\Cov(\zeta_0\,,\zeta_j)\sum_{k\in\Z^d}
		\varphi(k/n)\varphi((j+k)/n).
	\]
	This implies the first assertion of \eqref{UI}, since the Cauchy--Schwarz inequality
	for the  counting measure implies that
	$n^{-d}\sum_{k\in\Z^d}|\varphi(k/n)\varphi((j+k)/n)|
	\leq n^{-d}\sum_{k\in\Z^d} |\varphi(k/n)|^2$.
	Moreover, since $\varphi$ is piecewise continuous with compact support,
	\begin{align}\label{Riemann}
		\lim_{n\to\infty} n^{-d}\sum_{k\in\Z^d}
		\varphi(k/n)\varphi((j+k)/n) = \|\varphi\|_{L^2(\R^d)}^2
		\qquad\text{boundedly, for every $j\in\Z^d$}.
	\end{align}
	Therefore, the remaining assertion of \eqref{UI} follows from the dominated convergence theorem.
\end{proof}

\begin{lemma}\label{lem:dense}
	Let $\mathscr{D}$ be a dense subset of $\mathscr{C}$  in $L^2(\mathbb{R}^d)$,
	and suppose $S_n(\psi)\xrightarrow{\text{\rm d\,}} \sigma W(\psi)$, as $n\to\infty$,
	for every $\psi\in\mathscr{D}$.
	Then, $\{S_n(\varphi);\, \varphi\in\mathscr{C}\}\xrightarrow{\text{\rm fdd\,}}
	\{\sigma W(\varphi);\,\varphi\in\mathscr{C}\}$
	as $n\to\infty$.
\end{lemma}

\begin{proof}
	Choose and fix an arbitrary function $\varphi\in\mathscr{C}$. We aim to prove
	that $S_n(\varphi)\xrightarrow{\text{\rm d\,}} W(\varphi)$ as $n\to\infty$.
	The convergence of finite-dimensional distributions follows from this and the linearity of
	$\varphi\mapsto S_n(\varphi)$ and $\varphi\mapsto W(\varphi)$.
	For every $\varepsilon>0$ there exists $\psi\in\mathscr{D}$ such that
	$\|\varphi-\psi\|_{L^2(\R^d)}^2<\varepsilon$.
	Since $S_n(\psi)\xrightarrow{\rm d\,}\sigma W(\psi)$ as $n\to\infty$,
	we can write, for all $\xi \in \mathbb{R}$,
	\begin{align*}
		\limsup_{n\to\infty}\left| \E\e^{i\xi S_n(\varphi)} -
			\E\e^{i\xi \sigma W(\varphi)}\right|^2
			&\le {3}\limsup_{n\to\infty}\left| \E\e^{i\xi S_n(\varphi)}
			- \E\e^{i\xi S_n(\psi)}\right|^2
			+ { 3}\left| \E\e^{i\xi \sigma W(\varphi)} - \E\e^{i\xi \sigma W(\psi)}\right|^2\\
		&\le {3}\xi^2\limsup_{n\to\infty}\E\left( |S_n(\varphi)-S_n(\psi)|^2\right) +
			{3}\xi^2\sigma^2
			\E\left(|W(\varphi)-W(\psi)|^2\right),
	\end{align*}
	since $|\e^{iz}-\e^{iy}|\le|z-y|$ for all $z,y\in\R$. Property \eqref{UI}
	now ensures that, for all $\xi\in\R$,
	\[
		\limsup_{n\to\infty}\left| \E\e^{i\xi S_n(\varphi)} - \E\e^{i\xi \sigma W(\varphi)}\right|^2
		\le { 6}\xi^2\sigma^2\varepsilon\qquad\text{for every $\xi\in\R$}.
	\]
	This completes the proof since the left-hand side does not depend on $\varepsilon$.
\end{proof}

\begin{lemma}\label{lem:AI}
	Let $\delta>0$, and suppose that $(S_n(\psi_1)\,,S_n(\psi_2))$
	are asymptotically independent as $n\to\infty$ for all
	 $\psi_1,\psi_2\in\mathscr{U}$
	 such that ${\rm sep}({\rm supp}[\psi_1]\,,{\rm supp}[\psi_2]) \geq \delta$
	 $[$see \eqref{sep}$]$. Then,
	$(S_n(\varphi_1),\ldots,S_n(\varphi_N))$ are asymptotically independent as $n\to\infty$
	$[$see Definition \ref{def1}$]$
	for all $\varphi_1,\ldots,\varphi_N\in\mathscr{U}$ that satisfy
	\[
		{\rm sep}({\rm supp}[\varphi_i]\,,{\rm supp}[\varphi_j]) \geq \delta
		\qquad\text{for $1\leq i \neq j \leq N$ and integers $N\ge2$.}
	\]
\end{lemma}

\begin{proof}
	Set $\Phi_m := \varphi_1+\cdots+\varphi_m$ for
	all $m=1,\ldots,N$.
	For every $n\in\N$, let $D_{1,n}:=0$ and define
	\[
		D_{m,n}:= \bigg| \E\bigg[\prod_{j=1}^m\e^{i S_n(\varphi_j)} \bigg] -
		\prod_{j=1}^m\E \e^{iS_n(\varphi_j)} \bigg| =
		\bigg| \E\e^{i S_n(\Phi_m)} -
		\prod_{j=1}^m\E \e^{iS_n(\varphi_j)} \bigg|,
	\]
	for  $m=2\,,\ldots,N$.
	The linearity of $S_n$ and the triangle inequality together imply that for every $n\in\mathbb{N}$
	and $m=2,\ldots,N$,
	\begin{align*}
		D_{m,n}&=\left| \E\e^{i S_n(\Phi_m)} -
			\E\e^{iS_n(\Phi_{m-1})}\E\e^{iS_n(\varphi_m)}\right|
			+ \bigg| \E\e^{iS_n(\Phi_{m-1})} \E\e^{iS_n(\varphi_m)}-
			\prod_{j=1}^m\E\e^{iS_n(\varphi_j)}\bigg|\\
		&\le
		\left|\Cov\left[ \e^{iS_n(\Phi_{m-1})} ~,~
			\e^{-iS_n(\varphi_m)} \right]\right|  + D_{m-1,n}.
	\end{align*}
	Subtract $D_{m-1,n}$ from both sides and sum over $m$ to find that
	\[
		D_{N,n} \le \sum_{m=2}^N \left|\Cov\left[ \e^{iS_n(\Phi_{m-1})} ~,~
		\e^{iS_n(-\varphi_m)} \right]\right|.
	\]
	The separation condition on the $\varphi_j$'s implies
	that ${\rm sep}({\rm supp}[\Phi_{m-1}]\,, {\rm supp}[-\varphi_m]) \geq \delta$
	for all $m=2,\ldots,N$, and hence
	$\lim_{n\to\infty}D_{N,n}=0$ by asymptotic independence [see \eqref{AI}].
	Now relabel $\varphi_j$ as $\xi_j\varphi_j$, where $\xi_1,\ldots,\xi_N$ are arbitrary nonzero constants,
	in order to deduce that $S_n(\varphi_1),\ldots,S_n(\varphi_N)$ are asymptotically independent
	as $n\to\infty$. [This requires only the fact that $\xi_i\varphi_i$
	has the same  support as $\varphi_i$ for $i = 1, \ldots, N$.]
\end{proof}

\subsection{Proof of Theorem \ref{th:CLT}}\label{sec2.3}
In this section, we assume that the conditions of Theorem \ref{th:CLT} are met.

We first prove weak convergence in a special case, where the limit can be identified with
a Brownian motion.
Choose and fix real numbers $a_1$ and $a_2<b_2,\ldots,a_d<b_d$, and define
\[
	Q(r) = (a_1\,, a_1 + r] \times (a_2\,,b_2]\times\cdots
	\times(a_d\,,b_d]\qquad \text{for every $r \in(0\,,1]$}.
\]
Observe that $Q(r)\in\mathscr{U}$ for every $r\in(0\,,1]$.
For every $n \in\N$, we define one-parameter processes $Y_n$ and $Y$ as follows:
\[
	Y_n(r):= S_n(\mathbf{1}_{Q(r)})
	\quad\text{and}\quad
	Y(r) := \sigma W(\mathbf{1}_{Q(r)})\qquad \text{for every $r\in(0\,,1]$}.
\]
It is immediate that $Y$ is a one-dimensional Brownian motion with variance
$\sigma^2{\prod_{i=2}^d(b_i-a_i)}$. Our main objective is to prove the following specialized form of Theorem \ref{th:CLT}.

\begin{proposition}\label{pr:X}
	$Y_n \xrightarrow{\text{\rm fdd}}Y$ as $n\to\infty$.
\end{proposition}

Let us first deduce Theorem \ref{th:CLT} from its specialized form
Proposition \ref{pr:X}. The proposition will be verified subsequently.

\begin{proof}[Proof of Theorem \ref{th:CLT}]
	Let $\mathscr{D}:=\cup_{\delta>0}\mathscr{D}_\delta$,
	where for every $\delta>0$, $\mathscr{D}_{\delta}$
	denotes the collection of all functions  {$\psi\in\mathscr{U} $}
	that have the form,
	\begin{equation}\label{psi:rep}
		\psi=\psi_1+\cdots+\psi_m,
		\quad\text{where}\quad
		\psi_i = a_i\bm{1}_{Q_i},
	\end{equation}
	where $m\in\N$, $a_1,\ldots,a_m\in\R\setminus\{0\}$,  and
	$Q_1,\ldots,Q_m\in\mathscr{U}$ { are  upright boxes of the form
	$Q_i =(b^i_1, c^i_1] \times \cdots \times (b^i_d, c^i_d]$ for real numbers
	$b^i_1<c^i_1, \dots, b^i_d< c^i_d$, $i=1,\dots, m$, }
	and satisfy
	\[
		\text{\rm sep}\left( Q_i\,,Q_j\right)\geq \delta
		\quad\text{whenever $1\le i\neq j\le m$};
	\]
	see \eqref{sep}. Because $\mathscr{D}$ is dense in $L^2(\R^d)$, and hence also dense in $\mathscr{C}$,
	Lemma \ref{lem:dense} will imply Theorem \ref{th:CLT} once we prove that
	$S_{n}(\psi) \xrightarrow{\text{\rm d}\,}
	\sigma W(\psi)$, as $n\to\infty$, for every
	$\psi\in\mathscr{D}$.
	With this aim in mind, let us choose and fix some $\delta > 0$ and  $\psi\in\mathscr{D}_{\delta}$,
	and assume that $\psi$ has the representation \eqref{psi:rep}.
	By linearity, $S_{n}(\psi) = \sum_{i=1}^m
	S_{n}( \psi_i) =: \sum_{i=1}^m X_{i,n}$ a.s.,
	where $X_{i,n} := S_{n}(\psi_i)$.
	The asymptotic independence condition in Theorem \ref{th:CLT}
	and Lemma \ref{lem:AI} ensure that $\{X_{i,n}\}_{i=1}^m$
	describes an asymptotically independent sequence as $n\to\infty$;
	and Proposition \ref{pr:X}  implies that
	$X_{i,n}\xrightarrow{\text{d}\,}\sigma W(\psi_i)$
	as $n\to\infty$, for every $i=1,\ldots,m$.
	The asserted asymptotic independence then implies that
	$S_{n}(\psi) \xrightarrow{\text{d}\,}  Y_1+\cdots+Y_m$
	as $n\to\infty$,
	where $Y_1,\ldots,Y_m$ are independent, and the distribution of $Y_i$
	is the same as that of $\sigma W(\psi_i)$ for every $i=1,\ldots,m$.
	Because the supports of
	the $\psi_i$'s are disjoint, $W(\psi_1),\ldots,W(\psi_m)$
	are uncorrelated, hence independent, Gaussian random variables.
	In particular,  $S_{n}(\psi) \xrightarrow{\text{d}\,}
	\sigma W(\psi_1)+\cdots+\sigma W(\psi_m ) =\sigma W(\psi)$ as $n\to\infty$;
	see \eqref{psi:rep} for the last identity. This concludes the proof of
	Theorem \ref{th:CLT}.
\end{proof}

\begin{proof}[Proof of Proposition \ref{pr:X}]
	We will prove this proposition in  five steps.\\

	\noindent
	\textbf{ Step 1.} {\em The laws of $\{Y_n(r)\}_{n\geq1}$ are
	$L^2$-bounded uniformly in $r\in(0\,,1]$
	and $n\in\N$, and hence also tight uniformly over all $r\in(0\,,1]$. }

    	In order to see why, apply \eqref{UI} in order to see that
	\begin{equation}\label{eq:X:tight}
		\E\left( |Y_n(r)|^2\right) \leq \bar{\sigma}^2 n^{-d}
		\sum_{k\in\Z^d} \mathbf{1}_{Q(r)}(k/n)
		\le \bar{\sigma}^2 n^{-d}
		\sum_{k\in\Z^d} \mathbf{1}_{Q(1)}(k/n)
    	\end{equation}
    	for every $r \in (0\,, 1]$ and $n\in\N$. The final quantity in \eqref{eq:X:tight}
	is bounded uniformly
	in $n\in\N$ by $\bar{\sigma}^2$ times
	the upper Riemann sum of  $\mathbf{1}_{Q(1)}$, and the latter is finite.
   	This yields the desired $L^2$-boundedness, and tightness follows from Chebyshev's
	inequality.\medskip

	\noindent
	\textbf{Step 2.} {\em For every unbounded sequence $0<n_1<n_2<\cdots$ there exists a
	subsequence $n'=\{n_k'\}_{k=1}^\infty$ and random variables
	$\bar{Y}=\{\bar{Y}(r)\}_{r\in\mathbb{Q}\cap(0,1]}$ such that
	$Y_{n'_k}\xrightarrow{\text{\rm fdd}\,} \bar{Y}$ as $k\to\infty$.
	}

	This follows from uniform tightness in Step 1 and Cantor's diagonalization.\medskip

	\noindent
	\textbf{Step 3.} {\em $\bar{Y}$
	can be extended to a continuous process $\bar{Y}=\{\bar{Y}(r)\}_{r\in[0,1]}$.}

	Let $n$ tend to infinity along the subsequence
	$n'$, and appeal to Step 2, Fatou's lemma, and \eqref{UI} in order to
	see that for  every $R > r >0$
	with $r, R \in \mathbb{Q} \cap (0\,, 1]$,
	\[
		\E\left(|\bar{Y}(R) - \bar{Y}(r)|^2\right)
		\le\liminf_{k\to\infty}
		\E\left(|Y_{n'_k}(R) - Y_{n'_k}(r)|^2\right)
	 	=  \sigma^2|R - r|\prod_{i=2}^d(b_i-a_i).
	\]
	The continuity of $Y$ follows from this and
	Kolmogorov continuity theorem \cite[Theorem 2.8]{KS}.

	\medskip\noindent
	\textbf{Step 4.} \emph{We can realize $\bar{Y}=\{\bar{Y}(r)\}_{r\in[0,1]}$
	as an infinitely divisible process with stationary increments such that $\bar{Y}(0)=0$ and $\E[\bar{Y}(r)]=0$
	for all $r\in[0\,,1]$. Therefore, the process $\bar{Y}=\{\bar{Y}(r)\}_{r\in[0,1]}$ is
	a centered Brownian motion indexed and normalized
	such that $\Var[\bar{Y}(1)]=\sigma^2$.}

	We first prove that $\bar{Y}=\{\bar{Y}(r)\}_{r\in[0,1]}$
	is infinitely divisible.
	Let us choose and fix
	an integer $M\ge1$ and $M+1$  real numbers $0=:r_0 < r_1 < \cdots < r_M$.
	For every  sufficiently small $\delta > 0$, there exist rational points
	$r_{1, \delta}^\pm,\ldots,r_{M,\delta}^\pm\in\mathbb{Q}$ such that
	$0=r_0= r_{0, \delta}^+ < r_{1, \delta}^- <r_1 <
		r_{1, \delta}^+ < \cdots < r_{j, \delta}^- <r_j <  r_{j, \delta}^+
		< \cdots < r_{M, \delta}^- < r_M,$
	and $\delta \leq r_{j, \delta}^+ - r_{j, \delta}^- \leq 2\delta$
	for all $j = 1, \ldots, M-1$ and $r_M - r_{M, \delta}^- \leq \delta.$
	We choose and fix such a $\delta$ in order to deduce from
	the asymptotic independence condition (\text{\bf AI})
	and Lemma \ref{lem:AI} that
	$\{ Y_n(r_{j, \delta}^-)-Y_n(r_{j-1, \delta}^+)\}_{j=1}^{M}$
	are  asymptotically independent as $n\to\infty$.
	Hence, the random variables $\bar{Y}(r_{1, \delta}^-),
	\bar{Y}(r_{2, \delta}^-) - \bar{Y}(r_{1, \delta}^+),
	\ldots, \bar{Y}(r_{M, \delta}^-) - \bar{Y}(r_{M-1, \delta}^+)$
	are independent.
	Moreover, we may appeal to the continuity of
	$\bar{Y}=\{\bar{Y}(r)\}_{r\in[0,1]}$ (see Step 3) in order to
	conclude that for all ${\alpha}_1,\ldots,{\alpha}_M\in\R$,
	\begin{align*}
		\E\left[\e^{i\sum_{j=1}^M{\alpha}_j\left(\bar{Y}(r_j) - \bar{Y}(r_{j-1})\right)}\right]
			&= \lim_{\delta \to 0}\E\left[\e^{i\sum_{j=1}^M
			{\alpha}_j\left(\bar{Y}(r_{j, \delta}^-) - \bar{Y}(r_{j-1, \delta}^+)\right)}
			\right] \\
		& = \lim_{\delta \to 0}  \prod_{j = 1}^{M}\E\left[
			\e^{i{\alpha}_j\left(\bar{Y}(r_{j, \delta}^-) - \bar{Y}(r_{j-1, \delta}^+)\right)}\right]
			= \prod_{j = 1}^{M}\E\left[\e^{i{\alpha}_j\left(\bar{Y}(r_{j}) - \bar{Y}(r_{j-1})\right)}\right].
	\end{align*}
	It follows that the random variables
	$\bar{Y}(r_{1}), \bar{Y}(r_{2}) - \bar{Y}(r_{1}),\ldots, \bar{Y}(r_{M}) - \bar{Y}(r_{M-1})$
	are independent, whence $\bar{Y}=\{\bar{Y}(r)\}_{r\in[0,1]}$
	is infinitely divisible.

	Since $\zeta=\{\zeta_k\}_{k\in\Z^d}$ is stationary, the law of $Y_n(s)-Y_n(r)$
	is the same as the distribution of $Y_n(s-r)$ whenever $0<r<s$.
	Thus, we see that the distribution of
	{$\bar{Y}(s)-\bar{Y}(r)$}  is the same  as that of $\bar{Y}(s - r)$
	whenever $0<r<s$ are rational. The continuity of $\bar{Y}$ now ensures
	that the preceding holds in fact whenever $0<r<s$.
	This proves that the process $\bar{Y}=\{\bar{Y}(r)\}_{r\in\mathbb{Q}\cap[0,1]}$
	has stationary increments.

	Because of Step 1, and since $\E[Y_n(r)]=0$ for all $r\in(0\,,1]$ and $n\in\N$,
	it follows that $\E[\bar{Y}(r)]=0$ for all $r \in \mathbb{Q}\cap(0\,,1]$. A
	second appeal to Step 1 and continuity (Step 3) shows that $\E[\bar{Y}(r)]=0$
	for every $r\in[0\,,1]$.

	Finally, L\'evy's characteristic theorem of Brownian motion ensures that
	$\bar{Y}=\{\bar{Y}(r)\}_{r\in[0,1]}$ is
	a Brownian motion; see Bertoin \cite{Bertoin}.
	Therefore, it remains to check that $\Var[\bar{Y}(1)]=\sigma^2{\prod_{i=2}^d(b_i-a_i)}$.
	Indeed, Step 2 ensures that
	$Y_{n'_k}(1)\xrightarrow{\text{\rm d}\,} \bar{Y}(1)$ as $k\to\infty$.
	Therefore, the uniformly integrability
	condition \eqref{E:UI}  and Lemma \ref{UIequiv} imply that
	$\Var[\bar{Y}(1)] = \lim_{k \to\infty}\E[Y_{n'_k}^2(1)] = \sigma^2{\prod_{i=2}^d(b_i-a_i)}$, where the last
	equality is due to Lemma \ref{variance:limit}.


	\medskip\noindent
	\textbf{Step 5.}
	{\em We are ready to complete the proof of Proposition \ref{pr:X}.}

	So far, we have proved that for every unbounded {increasing} sequence $\{n_k\}_{k=1}^\infty$
	there exists a further unbounded   {increasing } subsequence $\{n'_k\}_{k=1}^\infty$ such that the finite-dimensional
	distributions of  $\{Y_{n'_k}(r)\}_{r \in \mathbb{Q}\cap (0, 1]}$
	converge to those of a Brownian motion $\bar{Y}$ as $k\to\infty$,
	and the speed of that Brownian motion  is always $\sigma^2{\prod_{i=2}^d(b_i-a_i)}$.
	In particular, the law of $\bar{Y}$ is the same as the law of
	$Y$ regardless of the choice of the original subsequence $\{n_k\}_{k=1}^\infty$. This proves that
	the finite-dimensional distributions of $\{Y_n(r)\}_{r\in\mathbb{Q}\cap(0\,,1]}$
	converge to those of $\{Y(r)\}_{r\in\mathbb{Q}\cap(0\,,1]}$.

	In order to conclude Proposition \ref{pr:X}, we need to show that for all integer
	$M\ge1$ and  for all $0 < r_1 < \cdots < r_M\le 1$,
	the characteristic function of
	$(Y_n(r_1)\,, \ldots, Y_n(r_M))$ converges to the characteristic
	function of $(Y(r_1)\,, \ldots, Y(r_M))$ as $n\to\infty$.
	For any $\varepsilon >0$, we can choose $R_1,\ldots,R_M\in\mathbb{Q}$ such that
	\begin{equation}\label{approx}
		r_k < R_k < r_k + \varepsilon \quad \text{and} \quad
		\|Y(r_k) - Y(R_k)\|_2 < \varepsilon \qquad \text{for all $k = 1, \ldots, M$}.
	\end{equation}
	Define
	$\mathscr{E}_n(\bm{\alpha}\,,\bm{\beta}):=\E[\e^{i\sum_{k=1}^M\alpha_kY_n(\beta_k)}]$
	and
	$\mathscr{E}(\bm{\alpha}\,,\bm{\beta}) := \E[\e^{i\sum_{k=1}^M\alpha_kY(\beta_k)}]$
	for all $\bm{\alpha}\in\R^M$ and $\bm{\beta}\in[0\,,1]^M$. Our goal is to prove that
	$\lim_{n\to\infty}\mathscr{E}_n(\bm{{\alpha}}\,,\bm{r})=\mathscr{E}(\bm{{\alpha}}\,,\bm{r})$
	for all $\bm{{\alpha}}\in\R^M$.  With this aim in mind, we can write
	\begin{align}\notag
		&\left|\mathscr{E}_n(\bm{{\alpha}}\,,\bm{r}) - \mathscr{E}(\bm{{\alpha}}\,,\bm{r}) \right|
			\le \left| \mathscr{E}_n(\bm{{\alpha}}\,,\bm{r}) - \mathscr{E}_n(\bm{{\alpha}}\,,\bm{R}) \right|
			+ \left|\mathscr{E}_n(\bm{{\alpha}}\,,\bm{R}) - \mathscr{E}(\bm{{\alpha}}\,,\bm{R}) \right|
			+ \left| \mathscr{E}(\bm{{\alpha}}\,,\bm{R}) - \mathscr{E}(\bm{{\alpha}}\,,\bm{r}) \right| \\\notag
		&\quad \leq \sum_{k = 1}^M |{\alpha}_k|\, \|Y_n(r_k) - Y_n(R_k)\|_2
			+ \left| \mathscr{E}_n(\bm{{\alpha}}\,,\bm{R}) - \mathscr{E}(\bm{{\alpha}}\,,\bm{R}) \right|
			+  \sum_{k = 1}^M |{\alpha}_k|\, \|Y(R_k) - Y(r_k)\|_2  \\
		&\quad \leq \sum_{k = 1}^M |{\alpha}_k|\, \|Y_n(r_k) - Y_n(R_k)\|_2
			+ \left| \mathscr{E}_n(\bm{{\alpha}}\,,\bm{R}) -
			\mathscr{E}(\bm{{\alpha}}\,,\bm{R}) \right| + \varepsilon \sum_{k = 1}^M |{\alpha}_k|,
			\label{final}
	\end{align}
	where the last inequality follows from \eqref{approx}.
	Since $\{Y_n(r)\}_{ r\in\mathbb{Q}\cap(0,1]}
	\xrightarrow{\text{\rm fdd\,}}
	\{Y(r)\}_{r\in\mathbb{Q}\cap(0,1]}$,
	the middle term in \eqref{final}
	vanishes as $n\to\infty$. Therefore,  \eqref{UI} implies that
	\begin{align*}
		\limsup_{n\to\infty}\left|\mathscr{E}_n(\bm{{\alpha}}\,,\bm{r}) - \mathscr{E}(\bm{{\alpha}}\,,\bm{r}) \right|
		&\leq  \sigma{\prod_{i=2}^d(b_i-a_i)^{1/2}} \sum_{k = 1}^M|{\alpha}_k| (R_k - r_k)^{1/2}
		+  \varepsilon\sum_{k = 1}^M|{\alpha}_k| \\
& \leq (\sigma \sqrt{\varepsilon} {\prod_{i=2}^d(b_i-a_i)^{1/2}}+ \varepsilon)\sum_{k = 1}^M|a_k|.
	\end{align*}
	This concludes the proof of Proposition \ref{pr:X} since $\varepsilon>0$ is arbitrary.
\end{proof}

\section{CLT for infinitely-many interacting diffusions}\label{S:SPDE}

The central limit theorem in Theorem \ref{th:CLT} can be applied to infinitely-many interacting diffusion processes, as indicated in the title of this paper.

\subsection{Malliavin calculus}\label{subsec:MC}
We introduce some elements of Malliavin calculus in order to establish the central limit theorem for infinitely-many interacting diffusion processes.
Let $\mathcal{H}= L^2(\R_+ \times \Z^d)$, and
recall that the Gaussian family $\{ \eta(h)\}_{h \in \mathcal{H}}$ formed by the Wiener integrals
$\mathcal{H}\ni h\mapsto \eta(h)= \int_0^\infty\sum_{x \in \Z^d}  h(s\,,x)\,\d B_s(x)$
defines an  {\it isonormal Gaussian process}.
In this framework, we can develop the Malliavin calculus as has been done, for instance,
by Nualart \cite{Nualart}.

We denote by $D$ the derivative operator, and $\mathbb{D}^{1,2}$ the Gaussian Sobolev space
generated by all $F\in L^2(\Omega)$ with derivative $DF\in L^2(\Omega\,;\mathcal{H})$.
In accord with the Poincar\'e inequality,
\begin{equation} \label{Poincare:Cov}
	|\Cov(F\,, G)| \le \int_0^\infty\sum_{z \in \Z^d}
	\left\| D_{s,z } F  \right\|_2
	\left\| D_{s,z}G \right\|_2\,\d s
	\qquad\text{for all $F,G\in\mathbb{D}^{1,2}$}.
\end{equation}
We will use this inequality extensively in the sequel.

\subsection{Comments on the solution \texorpdfstring{$u_t(x)$}{u(t,x)}}

The general theory of stochastic PDEs indexed by LCA groups (see Khoshnevisan
and Kim \cite{KhK15}) implies
that we may write the solution to \eqref{SHE}  in
the following \emph{mild} form (variation of parameters): Almost surely for all $t>0$ and $x\in\Z^d$,
\begin{align}\label{mild}
	u_t(x) = 1 + \int_0^t\sum_{y\in\Z^d}
	\bm{p}_{t -s}(y-x)\Phi(u_s(y))\, \d B_s(y),
\end{align}
where, for every $r\ge0$ and $w\in\Z^d$,  { $\bm{p}_t(w)=\P\{X_t=w\}$}
for a continuous-time random walk
$X=\{X_t\}_{t\ge0}$ { starting form the origin} on $\Z^d$ whose generator is $L$.
It also follows from general theory \cite{KhK15}
that for every real number $k\ge2$ there exists a positive real number $L=L(k\,,\Phi)$
such that
\begin{equation}\label{moments}
	\sup_{x\in\Z^d}\E\left( |u_t(x)|^k\right) \le L\e^{Lt}
	\quad \text{and} \quad
	\adjustlimits\sup_{n \geq 0}\sup_{x\in\Z^d}\E\left( |u_n(t\,,x)|^k\right) \le L\e^{Lt}
	\qquad\text{for all $t\ge0$},
\end{equation}
where $u_n$ denote the $n$th stage of the Picard iteration
approximation of $u$; see \eqref{Picard} below.

Finally, let us record the following elementary fact.

\begin{proposition}\label{stationarity}
	The random field $\{u_t(x): x \in\Z^d\}$ is stationary for every $t\ge0$.
\end{proposition}

We omit the proof as it follows along the same lines as in the proof of
\cite[Lemma 7.1]{CKNP}, using the fact that the law of
space-time white noise on $\R_+ \times \Z^d$ is  translation invariant.

\subsection{The Malliavin derivative of \texorpdfstring{$u_t(x)$}{u(t,x)}}

The following is the main result of this section.

\begin{proposition}\label{derivative:estimate}
	$u_t(x) \in \cap_{k\ge2}\mathbb{D}^{1, k}$ for all $(t\,, x) \in \R_+ \times \Z^d$.
	Furthermore, for all real numbers $T>0$ and $k\ge2$ there exists a number $C=C_{k,T,\Phi}>0$
	such that
	\begin{align}\label{Du(t,x)}
		\left\| D_{s, y}u_t(x)\right\|_k\le C \bm{p}_{t - s}(x - y) 
		\quad\text{for all $0<s<t<T$ and
		$x,y\in \Z^d$.}
	\end{align}
\end{proposition}

The proof of Proposition \ref{derivative:estimate} rests on the following sub-semigroup
property of the squares of the transition functions of the underlying random walk.

\begin{lemma}\label{lem:key}
	$\sum_{y \in \Z^d}[\bm{p}_t(x -y)\bm{p}_s(y -z) ]^2
	\leq [ \bm{p}_{t+s}(x -z)]^2$ for all $t, s \geq 0$ and $x,  z \in \Z^d$.
\end{lemma}

\begin{proof}
	We may change variables in order to reduce the problem to the case that $z=0$.
	Now suppose $z=0$, and
	let $X$ and $X'$ denote two independent copies of the random walk whose generator is $L$. We
	may observe that
	\begin{align*}
		\sum_{y \in \Z^d} \left[ \bm{p}_t(x -y)\bm{p}_s(y) \right]^2
			&= \sum_{y\in\Z^d} \P\left\{ X_{t+s}-X_s=x-y\,,X_{t+s}'-X_s'=x-y\,,X_s=y\,,X_s'=y\right\}\\
			&= \P\left\{ X_{t+s}=x\,,X_{t+s}'=x\,,X_s=X_s'\right\}.
	\end{align*}
	Drop the event $\{X_s=X_s'\}$ from the above to increase the latter probability to
	$[\bm{p}_{t+s}(x)]^2$; this yields the lemma.
\end{proof}

\begin{proof}[Proof of Proposition \ref{derivative:estimate}]
	We will provide the details for the case that $k=2$
	and merely point out the key part that needs a small revision
	to extend the proof to $k>2$.
	
	The proof is carried out in a few relatively direct steps,
	similar to the proof of \cite[Lemma 4.2]{CKNP_c};
	see also \cite[Theorem 6.4]{CKNP}.

	{\medskip\noindent\bf Step 1.}\
	\emph{Let $u_n$ denote the $n$th stage of the Picard iteration approximation of
	$u$. That is, $u_0(t\,,x) :=1$ and
	\begin{equation}\label{Picard}
		u_{n+1}(t\,,x) := 1 + \int_0^t\sum_{y\in\Z^d}
		\bm{p}_{t-s}(x-y)\Phi(u_n(s\,,y))\,\d B_s(y)
		\qquad[n\in\Z_+,t>0,x\in\Z^d].
	\end{equation}
	We claim that for all $n\in\N$,  $t>0$, $s\in(0\,,t)$, and $x,y\in\Z^d$,
	\[
		\E\left(| D_{s,y} u_{n+1}(t\,,x) |^2\right)\le
		[\bm{p}_{t-s}(x-y)]^2\bigg\{  { K\e^{L s}
		\sum_{\nu=0}^{n-1} 2^{\nu+1}\lip_\Phi^{2\nu} \frac{(t-s)^\nu}{\nu!}
		+ 2^{n+1}\lip_\Phi^{2n} \frac{(t-s)^{n}}{n!} | \Phi(1)|^2} \bigg\},
	\]
	where $K$ is a positive real that depends only on $\Phi$. Furthermore,
	\begin{equation}\label{Du_1}
		\E\left(| D_{s,y} u_1(t\,,x) |^2\right) = [\bm{p}_{t-s}(x-y)]^2|\Phi(1)|^2,
	\end{equation}
	for every $t>0$, $s\in(0\,,t)$, $x,y\in\Z^d$.\medskip
	}

	\noindent{\sc Proof of Step 1.}\
	We apply the properties of the divergence operator (see \cite[Prop.\ 1.3.8]{Nualart})
	in order to find from \eqref{Picard} that: (1)
	$D_{s,y} u_1(t\,,x) = \bm{p}_{t-s}(x-y)\Phi(1)$, which proves \eqref{Du_1};
	and (2)
	\begin{align}\label{derivative}
		D_{s, y}u_{n + 1}(t\,,x) = \bm{p}_{t - s}(x - y)\Phi(u_n(s\,,y))
		+ \int_s^t\sum_{z\in\Z^d}
		\bm{p}_{t - r}(x - z)D_{s, y}\Phi(u_n(r\,,z))\,\d B_r(z),
	\end{align}
	for every $n\in\N$, $(t\,,x)\in(0\,,\infty)\times\Z^d$, and $(s\,,y)\in(0\,,t)\times\Z^d$.
	Because $|\Phi(z)|\le|\Phi(0)|+\lip_\Phi|z|$ for all $z\in\R$, the preceding yields the following
	bounds:
	\begin{align}
		\label{E:SecDNorm}
		&\left\| D_{s,y} u_{n+1}(t\,,x) \right\|_2\\ \notag
		&\le \bm{p}_{t-s}(x-y)\left( |\Phi(0)|+\lip_\Phi\|u_n(s\,,y)\|_2\right)
			+ \bigg( \int_s^t \sum_{z\in\Z^d} [\bm{p}_{t-r}(x-z)]^2
			\| D_{s,y}\Phi(u_n(r\,,z))\|_2^2\, \d r \bigg)^{1/2}\\ \notag
		&\le \bm{p}_{t-s}(x-y)\left( |\Phi(0)|+\lip_\Phi\|u_n(s\,,y)\|_2\right)
			+ \lip_\Phi\bigg( \int_s^t \sum_{z\in\Z^d} [\bm{p}_{t-r}(x-z)]^2
			\| D_{s,y} u_n(r\,,z) \|_2^2\, \d r \bigg)^{1/2},
	\end{align}
	thanks to the chain rule of Malliavin calculus for Lipschitz-continuous
	functions (see Nualart \cite[Proposition 1.2.4]{Nualart}).
	In order to adapt the preceding the asserted bound for $\| D_{s, y}u_t(x)\|_k$
	when $k>2$, we apply the Burkholder-Davis-Gundy inequality instead, and obtain
	\[
		\left\| D_{s, y}u_t(x)\right\|_k
		\le  \bm{p}_{t-s}(x-y)\left( |\Phi(0)|+\lip_\Phi\|u_n(s\,,y)\|_2\right)
		+ A\bigg( \int_s^t \sum_{z\in\Z^d} [\bm{p}_{t-r}(x-z)]^2
		\| D_{s,y} u_n(r\,,z) \|_k^2\, \d r \bigg)^{1/2},
	\]
	where $A=A(k\,,\lip_{\Phi})$. We continue with the case $k=2$ from now on, but 
	point out that we obtain the general form of the proposition by keeping
	track of the effect of using the above modification.
	
	Recall  \eqref{moments} and let
	$K := (|\Phi(0)| + \sqrt{L}\,\lip_\Phi )^2$
	in order to find that
	\begin{align}\label{E2:IntIneq}
		\left\| D_{s,y} u_{n+1}(t\,,x) \right\|_2^2
		\le 2K\e^{Ls}[\bm{p}_{t-s}(x-y) ]^2+
		2\lip_\Phi^2\int_s^t\sum_{z\in\Z^d}[\bm{p}_{t-q}(x-z)]^2
		\|D_{s,y}u_n(q\,,z)\|_2^2\,\d q,
	\end{align}
	where we have used the elementary inequality $(a+b)^2\le2a^2+2b^2$, valid for every $a,b\in\R$.
	In particular, we may freeze the variables $s$ and $y$ in order to see that the following functions
	$g_1,g_2,\ldots,$ defined via
	\begin{equation}\label{g_n}
		g_n(r\,,x) := \| D_{s,y} u_n(s+r\,,x+y) \|_2^2\qquad[n\in\N,
		r>0,x\in\Z^d],
	\end{equation}
	satisfy
	\begin{equation}\label{g_1:UB}
		g_1(r\,,x) \le [\Phi(1)\bm{p}_r(x)]^2,
	\end{equation}
	and
	\[
		g_{n+1}(r\,,x) \le 2K\e^{L s}[\bm{p}_r(x)]^2 +
		2\lip_\Phi^2\int_0^r\d q \sum_{z\in\Z^d} [\bm{p}_{r-q}(x-z)]^2 g_n(q\,,z),
	\]
	for every $n\in\N$, $r>0$, and $x\in\Z^d$. We may iterate this recursive inequality
	once in order to see that if $n\ge 2$ is an integer, $r>0$, and $x\in\Z^d$, then
	\begin{align*}
		g_{n+1}(r\,,x) \le &2K\e^{L s}[\bm{p}_r(x)]^2 +
			4K { \e^{L s}}
			\lip_\Phi^2\int_0^r\d q \sum_{z\in\Z^d} [\bm{p}_{r-q}(x-z)]^2 [\bm{p}_q(z)]^2\\
		&+ 4\lip_\Phi^4\int_0^r\d q_1\sum_{z_1\in\Z^d}[\bm{p}_{r-q_1}(x-z_1)]^2
			\int_0^{q_1}\d q_2\sum_{z_2\in\Z^d}[\bm{p}_{q_1-q_2}(z_1-z_2)]^2 g_{n-1}(q_2\,,z_2)\\
		&\hskip-1in\le 2K\e^{L s}[\bm{p}_r(x)]^2 +
			4K\e^{L s} \lip_\Phi^2 r[\bm{p}_r(x)]^2
			+ 4\lip_\Phi^4\int_0^r\d q_1
			\int_0^{q_1}\d q_2\sum_{z\in\Z^d}[\bm{p}_{r-q_2}(x-z)]^2 g_{n-1}(q_2\,,z),
	\end{align*}
	owing to sub-semigroup property of $\bm{p}^2$ [Lemma \ref{lem:key}].
	We may repeat once again to see that
	if $n\ge 3$ is an integer, $r>0$, and $x\in\Z^d$, then
	\begin{align*}
		g_{n+1}(r\,,x) \le & 2K\e^{L s}[\bm{p}_r(x)]^2 +
			4K\e^{L s} \lip_\Phi^2 r[\bm{p}_r(x)]^2\\
		& + 8K\e^{L s}\lip_\Phi^4\int_0^r\d q_1
			\int_0^{q_1}\d q_2\sum_{z\in\Z^d}[\bm{p}_{r-q_2}(x-z)]^2
			[\bm{p}_{q_2}(z)]^2\\
		&+ 16\lip_\Phi^6\int_0^r\d q_1
			\int_0^{q_1}\d q_2\sum_{z\in\Z^d}[\bm{p}_{r-q_2}(x-z)]^2
			\int_0^{q_2}\d q_3
			\sum_{z'\in\Z^d}[\bm{p}_{q_2-q_3}(z-z')]^2 g_{n-2}(q_3\,,z'),
	\end{align*}
	and so on. Continue this iteration process and deduce from \eqref{g_n} the asserted
	inequality for $\E(| D_{s,y} u_{n+1}(t\,,x) |^2)$
	after a change of variables $[s+t\leftrightarrow s+r$ and $x+y\leftrightarrow x]$,
	using the simple fact that $\sum_{w\in\Z^d}[\bm{p}_\tau(w)]^2\le1$
	for all $\tau\ge0$ in order
	to obtain the last term in the curly brackets.
	This and \eqref{g_1:UB} together prove
	the validity of Step 1.\qed\medskip

	\noindent{\bf Step 2.} \emph{$\sup_{n\in\Z_+}
	\E(\| D u_n(t\,,x)\|_{\mathcal{H}}^2)<\infty$
	for all $t\ge0$ and $x\in\Z^d$, where $\mathcal{H}= L^2(\R_+ \times \Z^d)$
	was defined in \S\ref{subsec:MC}.
	}\medskip

	\noindent{\sc Proof of Step 2.}\
	In accord with   \cite[Corollary 1.2.1]{Nualart}, $D_{s,y} u_n(t\,,x)=0$ when $s\ge t$. Therefore, Step 1 implies
	that $\E( \| Du_n(t\,,x)\|_{\mathcal{H}}^2)
	\le \int_0^t c_n(s\,,t)\| \bm{p}_s\|_{\ell^2(\Z^d)}^2\,\d s$
	for all $(t\,,x)\in\R_+\times\Z^d$ and $n\in\N$, where $c_1\equiv |\Phi(1)|^2$ and
	\begin{equation}\label{c_n}
		c_{n+1}(s\,,t) = K\e^{L s}   {
		\sum_{\nu=0}^{n-1} 2^{\nu+1}\lip_\Phi^{2\nu} \frac{(t-s)^\nu}{\nu!}
		+ 2^{n+1}\lip_\Phi^{2n} \frac{(t-s)^{n}}{n!} | \Phi(1)|^2}.
	\end{equation}
	Step 2 is a consequence of the above and  the elementary fact that
	$\sup_{n\in\N}c_n(s\,,t)<\infty$.\qed\\

	Finally, we complete the proof of Proposition \ref{derivative:estimate}
	in a third, and final, step.\medskip

	\noindent{\bf Step 3.}\
	\emph{ $u_t(x)\in \mathbb{D}^{1,k}$ for every $k\ge2$ and
	$(t\,,x)\in\R_+\times\Z^d$,
	and \eqref{Du(t,x)} holds for the parameter dependencies of
	Proposition \ref{derivative:estimate}.}
	\medskip

	\noindent{\sc Proof of Step 3.}\
	Again we consider only the case $k=2$; the general case is proved similarly.
	Since $u_0\equiv1$, Step 3 has content only when $t>0$. With this comment in mind,
	let us choose and fix $t>0$ and $x\in\Z^d$. General theory \cite{KhK15} ensures
	that $\lim_{n\to\infty}u_n(t\,,x) = u_t(x)$ in $L^2(\Omega)$ for every
	$t\ge0$ and $x\in\Z^d$. Therefore,
	the closeability properties of the Malliavin derivative (see Nualart \cite[Lemma 1.2.3]{Nualart})
	and Step 2 together imply that
	$Du_n(t\,,x)\to Du(t\,,x)$, as $n\to\infty$, in the weak topology of
	{$L^2(\Omega\, ; L^2(\R_+\times \Z^d))$}; and moreover, that $u_t(x)\in\mathbb{D}^{1,2}$.
	Now, we apply Cantor's diagonalization in order to see
	that there exists an unbounded sequence $\{n(\ell)\}_{\ell=1}^\infty$
	of positive integers such that for every $y \in \Z^d$,
	$D_{\bullet, y}u_{n(\ell)}(t\,,x) \to D_{\bullet, y}u(t\,,x)$,
	as $\ell\to\infty$, in the weak topology of
	$L^2(\Omega\, ; L^2(\R_+))$.
	We next use a bounded and smooth approximation $\{\psi_\varepsilon\}_{\varepsilon>0}$
	to the identity in $\R_+$,  and apply Fatou's lemma and
	the self-duality of $L^2$ spaces in order
	to find that
	\begin{equation}\label{ED}\begin{split}
		\|D_{s,y}u_t(x) \|_2 & \le  \liminf_{\varepsilon\downarrow 0}
			\left \| \int_0^s D_{s'\!,y} u_t(x)
			\psi_\varepsilon(s-s')\, \d s' \right\|_2\\
		& =  \liminf_{\varepsilon\downarrow 0}
			\sup_{\|G \|_2\le 1}
			\left| \int_0^s  \E\left[ G D_{s'\!,y} u_t(x) \right]
			\psi_\varepsilon(s-s')\, \d s' \right|,
	\end{split}\end{equation}
	for all $y \in \Z^d$ and almost every $s \in (0\,,t)$.
	Choose and fix a random variable $G\in L^{2}(\Omega)$ such that
	$\E(|G|^{2})\le 1$. For all $y \in \Z^d$,
	$D_{\bullet, y}u_{n(\ell)}(t\,,x) \to D_{\bullet, y}u_t(x)$,
	as $\ell\to\infty$, in the weak topology of
	$L^2(\Omega\, ; L^2(\R_+))$.
	Thus, we find that for all $y \in \Z^d$ and almost all $s \in (0\,, t)$,
	\begin{align*}
		\left| \int_0^s  \E \left[ G D_{s'\!,y} u_t(x) \right]
			\psi_\varepsilon(s-s')\, \d s' \right|
			&= \lim _{\ell\rightarrow \infty} \left|
			\int_0^s  \E\left[ G D_{s'\!,y} u_{n(\ell)}(t\,,x) \right]
			\psi_\varepsilon(s-s')\, \d s'   \right| \\
		&\le\limsup_{\ell\to\infty}
			\int_0^s  \left\| D_{s',y} u_{n(\ell)}(t\,,x) \right\|_2
			\psi_\varepsilon(s-s')\, \d s'\\
		&\le \lim_{\ell\to\infty} \sqrt{c_{n(\ell)}(s\,,t)}
			\int_0^s  \bm{p}_{t - s'}(x - y)
			\psi_\varepsilon(s-s')\, \d s',
	\end{align*}
	owing to Step 1,
	where $c_n$ was defined in \eqref{c_n}. Letting $\varepsilon\to 0$
	to deduce the result from \eqref{ED}, as well as
	the boundedness and the continuity of $s\mapsto \bm{p}_{t-s}(x-y)$ for every $t>0$ and $x,y\in\Z^d$.
\end{proof}

\subsection{Proof of Theorem \ref{CLT:SHE}}

Let us make a small observation before
we begin the proof of Theorem \ref{CLT:SHE}:
Thanks to the Poincar\'e inequality \eqref{Poincare:Cov}
and the chain rule of Malliavin calculus \cite[Proposition 1.2.4]{Nualart},
\[
|\Cov[g(u_t(0)) \,, g(u_t(x))]| \le \lip_g^2\int_0^t\sum_{z \in \Z^d}
\| D_{s,z } u_t(0)  \|_2\| D_{s,z} u_t(x) \|_2\,\d s.
\]
Therefore, Proposition \ref{derivative:estimate}  yields
\begin{equation}\label{SumCov}
	\sum_{x\in\Z^d}
	\left|\Cov\left[g(u_t(0)) \,, g(u_t(x))\right]\right|
	\le A\int_0^t
	\sum_{x,z\in\Z^d} \bm{p}_{t-s}(-z)\bm{p}_{t-s}(x-z)\,\d s<\infty,
\end{equation}
for a real number $A=A(K\,,\lip_g\,,t\,,\lip_\Phi\,,L).$
{We now proceed to the proof of Theorem \ref{CLT:SHE}.}

\begin{proof}[Proof of Theorem \ref{CLT:SHE}]
	Choose and fix some $t>0$ throughout,
	and define
	\[
		\zeta_k := g(u_t(k)) - \E[g(u_t(0))]\qquad
		\text{for every $k \in \Z^d$.}
	\]
	By Proposition \ref{stationarity}
	and \eqref{moments}, $\{\zeta_k\}_{k \in\Z^d}$ is stationary,
	$\E[\zeta_0] = 0$, and $\Var(\zeta_0)< \infty$. Furthermore, \eqref{SumCov}
	assures us that
	$\sigma_{g,t}^2=\sum_{k\in\Z^d}
	\Cov(\zeta_0\,,\zeta_k)= \sum_{k\in\Z^d}
	\Cov(g(u_t(0))\,,g(u_t(k)))$
	is an absolutely convergent sum.
	We verify uniform integrability \eqref{E:UI} next.
	
	For every $\varphi \in \mathscr{C}$ define, following \eqref{S_n},
	\begin{align} \label{E:S2}
		S_n(\varphi) &= 
			{ {n^{-d/2}}}\sum_{k\in\Z^d}\left\{
			g(u_t(k))- \E[g(u_t(0))]\right\}\varphi(k/n)\\
		&= n^{-d/2} \sum_{k\in\Z^d}\bigg( \sum_{y\in\Z^d}
			\int_0^t \E\left[ D_{s,y}[g(u_t(k))] \mid \mathcal{F}_s\right]
			\d B_s(y)\bigg)\varphi(k/n),
	\end{align}
	where $\mathcal{F}_s:=$ the $\sigma$-algebra generated by
	$\{B_r(y);\, y\in\Z^d, r\in[0\,,s]\}$,
	and we have used the Clark--Ocone formula in the last line. We apply Minkowski's inequality
	and the Burkholder-Davis-Gundy inequality in order to see from the above that
	for all $p\ge2$ there exists $c_p>0$ such that for every $n\in\N$,
	\begin{align*}
		\| S_n(\varphi)\|_p^2 &\le \frac{c_p}{n^d}\sum_{y\in\Z^d}\sum_{k,k'\in\Z^d}
			|\varphi(k/n)\varphi(k'/n)|\int_0^t
			\left\| \E\left[ D_{s,y}[g(u_t(k))] \mid \mathcal{F}_s\right]
			\E\left[ D_{s,y}[g(u_t(k'))] \mid \mathcal{F}_s\right] 
			\right\|_{p/2}\,\d s\\
		&\le \frac{c_p}{n^d}\sum_{y\in\Z^d}\sum_{k,k'\in\Z^d}
			|\varphi(k/n)\varphi(k'/n)|\int_0^t
			\left\| D_{s,y}[g(u_t(k))]\right\|_p
			\left\| D_{s,y}[g(u_t(k'))] \right\|_p\,\d s,
	\end{align*}
	the last line valid thanks 
	to the Cauchy-Schwarz inequality and Jensen's inequality for conditional expectations.
	Thus, the chain rule of Malliavin derivative (see \cite[Proposition 1.2.4]{Nualart}) yields
	\[
		\| S_n(\varphi)\|_p^2 \le \frac{c_p\lip_g^2}{n^d}\sum_{y\in\Z^d}\sum_{k,k'\in\Z^d}
		|\varphi(k/n)\varphi(k'/n)|\int_0^t
		\left\| D_{s,y}u_t(k) \right\|_p
		\left\| D_{s,y}u_t(k')\right\|_p\,\d s.
	\]
	Proposition \ref{derivative:estimate} can now be used to deduce that,
	uniformly for all $n\in\N$,
	\begin{align*}
		\| S_n(\varphi)\|_p^2 &\le \frac{\text{const}}{n^d}\cdot\sum_{y\in\Z^d}\sum_{k,k'\in\Z^d}
			|\varphi(k/n)\varphi(k'/n)|\int_0^t
			\bm{p}_{t-s}(k-y)\bm{p}_{t-s}(k'-y)\,\d s\\
		&=\frac{\text{const}}{n^d}\cdot\sum_{k,k'\in\Z^d}
			|\varphi(k/n)\varphi(k'/n)|\int_0^t
			\bm{p}_{2(t-s)}(k-k')\,\d s
			\hskip.5in\text{[semigroup property]}\\
		&\leq \frac{\text{const}}{n^d}\cdot\sum_{k\in\Z^d}
			|\varphi(k/n)|^2\int_0^t
			\sum_{j\in\Z^d}\bm{p}_{2(t-s)}(j)\,\d s
			\hskip.5in\text{[Cauchy-Schwarz inequality]}.
	\end{align*}
	Since $\bm{p}_{2(t-s)}(j)$ sums up to $1$, the integral is equal to $t$, whence
	$\sup_{n\in\N}\|S_n(\varphi)\|_p<\infty$
	for every $\varphi\in\mathscr{C}$. Because $p>2$,
	the uniform integrability condition \eqref{E:UI} follows.

	In light of \eqref{SumCov} and  Theorem \ref{th:CLT}
	it remains to prove that $\{S_n(\varphi);\, \varphi\in\mathscr{C}\}$
	satisfies the asymptotic independence condition (\textbf{AI}).
	First, note that for all $z \in \Z^d$ and almost every $s\in (0\,, t)$,
	\[
		D_{s, z}S_n(\varphi) =
		n^{-d/2}\sum_{k\in\Z^d}D_{s, z}g(u_t(k))\varphi(k/n)
		= n^{-d/2}\sum_{k\in\Z^d}g'(u_t(k))D_{s, z}u_t(k)\varphi(k/n)
		\quad\text{ a.s.}
	\]
	Therefore, we once again
	envoke the chain rule of Malliavin derivative (see Nualart \cite[Proposition 1.2.4]{Nualart})
	in order to see that for all 
	$p\ge 2$, $n\in\N$, and $z \in \Z^d$, and for almost every $s\in (0\,, t)$,
	\begin{equation}\label{DS}
		\|D_{s, z}S_n(\varphi)\|_p
		\leq  \frac{A_{p,t}}{n^{d/2}}
		\sum_{k\in\Z^d}\bm{p}_{t -s}(k -z) |\varphi(k/n)|,
	\end{equation}
	for a number $A_{p,t}>0$, where the last inequality follows from Proposition \ref{derivative:estimate}.
	
	Choose and fix $a,b\in\R$, $\delta >0$, and $\varphi_1, \varphi_2 \in \mathscr{U}$ such that
	\begin{equation}\label{separate}
		{\rm sep}({\rm supp}[\varphi_1]\,,{\rm supp}[\varphi_2]) \geq \delta.
	\end{equation}
	Our remaining goal is to prove that
	\[
		\mathcal{C}_n := \left|\E\left[\e^{iaS_n(\varphi_1) + ibS_n(\varphi_2)}\right] -
		\E\left[\e^{iaS_n(\varphi_1)}\right]\E\left[\e^{ibS_n(\varphi_2)}\right]\right|
		\to { 0} \quad\text{as $n\to\infty$}.
	\]
	Since $\mathcal{C}_n= |\Cov(\e^{iaS_n(\varphi_1)}\,, \e^{-ibS_n(\varphi_2)})|,$
	\eqref{Poincare:Cov} and \eqref{DS}
	together imply that
	\begin{align*}
		\mathcal{C}_n&\leq |ab| \sum_{z \in \Z^d}\int_0^t
			\|D_{s, z}S_n(\varphi_1)\|_2 \|D_{s, z}S_n(\varphi_2)\|_2\, \d s  \\
		&\leq \frac{\text{const}}{n^d}
			\sum_{k, m, z \in \Z^d}
			|\varphi_1(k/n) \varphi_2(m/n)|\int_0^t
			\bm{p}_s(k -z)\bm{p}_s(m -z)\d s,
	\end{align*}
	uniformly in $n\in\N$.
	Observe that $\sum_{z\in\Z^d}\bm{p}_s(k-z)\bm{p}_s(m-z)=\P\{X_s-X_s'=k-m\}$
	where $X$ and $X'$ are i.i.d.\ copies of a random walk with generator $L$. Thus,
	we can re-index the sums to find that
	\begin{equation}\label{cov12}
		\mathcal{C}_n
		\le \text{const}\cdot
		\sum_{\ell\in \Z^d}\int_0^t \P\{X_s-X_s'=\ell\}\,\d s\
		\frac{1}{n^d}\sum_{m \in\Z^d}\left|
		\varphi_1 \left( \frac{m + \ell}{n}\right)\varphi_2\left( \frac{m}{n}\right)\right|.
	\end{equation}
	The Cauchy-Schwarz inequality and \eqref{Riemann} [with $j = 0$] together
	imply that the final quantity
	$n^{-d}\sum_{m\in\Z^d}|\varphi_1((m+\ell)/n)\varphi_2(m/n)|$
	in \eqref{cov12} is bounded uniformly over
	all $n\in\N$ and $m\in\Z^d$.
	Therefore, an appeal to the dominated convergence theorem assures us that
	\[
		\lim_{n\to\infty}\sum_{\ell\in \Z^d}\int_0^t \P\{X_s-X_s'=\ell\}\,\d s\cdot
		\frac{1}{n^d}\sum_{m \in\Z^d}
		\left| \varphi_1\left( \frac{m + \ell}{n}\right)\varphi_2\left( \frac{m}{n}\right)
		\right|
		= t\int_{\R^d}|\varphi_1(y)\varphi_2(y)|\,\d y  = 0,
	\]
	owing to \eqref{Riemann}
	and \eqref{separate}. This and \eqref{cov12} together imply
	that $\lim_{n\to\infty}\mathcal{C}_n=0$, and complete the proof of the theorem.
\end{proof}
{
\begin{remark}
An anonymous referee informs us a functional version of our Theorem \ref{CLT:SHE}. Consider $\varphi= \mathbf{1}_{Q(r)}$ (see the definition of $\mathbf{1}_{Q(r)}$ in \S\ref{sec2.3}). Then, as a process in $r\in [0, 1]$, the left-hand side of \eqref{CLT} converges in distribution to Brownian motion in the space $C[0, 1]$. This is guaranteed by \cite[Theorem 19.2]{Billingsley} since the moment estimate and asymptotic independence established in the proof of Theorem  \ref{CLT:SHE} verify the conditions of \cite[Theorem 19.2]{Billingsley}. Meanwhile, the same anonymous referee also informs us that when $d=1$, our Theorem \ref{th:CLT} can be deduced from  Theorem 4.5
of Jakubowski \cite{Jakubowski91}. 
\end{remark}
}

\noindent \textbf{Acknowledgement.}\
We thank an anonymous referee for patiently pointing out an oversight in an earlier version of this paper. 
This paper has benefitted from comments and suggestions by two anonymous referees.


\begin{thebibliography}{999}\small
\bibitem{Bertoin}
	Bertoin, J. (1996).
	{\it L\'evy Processes.}
	Cambridge University Press, Cambridge, UK.
\bibitem{Billingsley}	
Billingsley, P. (1968)
{\it Convergence of probability measures}.
 John Wiley \& Sons, Inc., New York-London-Sydney.
\bibitem{Bradley}
	Bradley, R. C. (2007).
	{\it Introduction to Strong Mixing Conditions, Vol.\ 1.}
	Kendrick Press, Heber City, Utah.
\bibitem{Bradley_PS}
	Bradley, R. C. (2005).
	Basic properties of strong mixing conditions. A survey and some open questions.
	Update of, and a supplement to, the 1986 original.
	{\it Probab.\ Surv.}\ {\bf 2} 107--144.
\bibitem{CM}
	Carmona, R. A. and  Molchanov, S. A. (1994).
	Parabolic Anderson Problem and Intermittency,
	\emph{Mem.\ Amer.\ Math.\ Soc.}, Providence, RI.
\bibitem{CKNP}
	Chen, L., Khoshnevisan, D., Nualart, D., and Pu, F. (2019).
	Spatial ergodicity for SPDEs via Poincar\'e-type inequalities. Preprint available
	at \url{https://arxiv.org/abs/1907.11553}.
\bibitem{CKNP_b}
	Chen, L., Khoshnevisan, D., Nualart, D., and Pu, F. (2019).
	Poincar\'e inequality, and central limit theorems for parabolic stochastic partial
	differential equations. Preprint available at
	\url{https://arxiv.org/abs/1912.01482}.
\bibitem{CKNP_c}
	Chen, L., Khoshnevisan, D., Nualart, D.,  and Pu, F. (2020).
	Spatial ergodicity and central limit theorem for parabolic Anderson
	model with delta initial condition.
	Preprint available at \url{http://arxivorg/abs/2005.10417}.
\bibitem{Deuschel}
	Deuschel, J.-D. (1988).
	Central limit theorem for an infinite lattice system of interacting diffusion processes.
	 {\it Ann.\ Probab.}\ {\bf 16}{\it (2)} 700--716.
\bibitem{EsaryProschanWalkup}
	Esary, J. D.,  Proschan, F., and  Walkup, D. W. (1967).
	Association of random variables, with applications.
	{\it Ann.\ Math.\ Statist.}\ {\bf 38} 1466--1474.
\bibitem{Federer}
	Federer, H. (1969).
	{\it Geometric Measure Theory}.
	Die Grundlehren der mathematischen Wissenschaften
	{\bf 153} Springer-Verlag, Berlin-Heidelberg-New York.
\bibitem{Ibragimov}
	Ibragimov, I. A. (1962).
	Some limit theorems for stationary processes.
	{\it Teor. Verojatnost.\ i Primenen.}\ {\bf 7} 361--392.
\bibitem{Jakubowski91}
	Jakubowski, A. (1991).
	{\it Asymptotic Independent Representations for Sums and Order Statistics of Stationary Sequences}.
	Available for download at
	\url{http://www-users.mat.uni.torun.pl/~adjakubo/hab.pdf}.
\bibitem{KS}
	Karatzas, I. and Steven E. S. (1991).
	{\it Brownian Motion and Stochastic Calculus.}
	Second edition.
	Graduate Texts in Mathematics, \textbf{113}. Springer-Verlag, New York
\bibitem{KhK15}
	Khoshnevisan, D.  and  Kim, K. (2015).
	Nonlinear noise excitation of intermittent stochastic PDEs and the topology of LCA groups.
	{\it Ann.\ Probab.}\ {\bf 43}{\it(4)} 1944--1991.
\bibitem{Lahiri}
	Lahiri, S. N. (2003).
	A necessary and sufficient condition for asymptotic independence of discrete Fourier transforms
	under short- and long-range dependence.
	{\it Ann.\ Statist.}\ {\bf 31}{\it (2)} 613--641.
\bibitem{MerlevedePeligradUtev}
	Merlev\`ede, F., Peligrad, M. and  Utev, S. (2006).
	Recent advances in invariance principles for stationary sequences.
	{\it Probab.\ Surv.}\ {\bf 3} 1--36.
\bibitem{NewmanWright}
	Newman, C. M. and Wright, A. L.  (1981).
	An invariance principle for certain dependent sequences.
	{\it Ann.\ Probab.}\ {\bf 9}{\it (9)} 361--371.
\bibitem{Nualart}
	Nualart, D. (2006).
	{\it  The Malliavin Calculus and Related Topics}. Springer, New York.
\bibitem{Rosenblatt}
	Rosenblatt, M. (1956).
	Central limit theorems for stationary processes.
	{\it Proc.\ Sixth Berkeley Symp.\ Probab.\  Statist.}\
	{\bf 2} 551--561. University of California Press, Los Angeles, 1972.
\bibitem{SS}
	Shiga, T. and Shimizu, A. (1980).
	{Infinite-dimensional stochastic differential equations and their applications}.
	{\it J. Math.\ Kyoto Univ.}\
	{\bf20}{\it 3} 395--416.
\end{thebibliography}
\end{document}